\documentclass[leqno,a4paper,12pt]{amsart}
\usepackage{amsmath,amssymb,amscd,amsthm,latexsym,amsfonts, epsfig, graphicx}
\usepackage{color}
\usepackage[bookmarks=true]{hyperref}   

\numberwithin{equation}{section}

\newtheorem{theorem}{Theorem}[section]

\newtheorem{lemma}[theorem]{Lemma}

\newtheorem{proposition}[theorem]{Proposition}

\theoremstyle{remark}
\newtheorem{remark}[theorem]{Remark}

\newcommand{\dv}{\operatorname{div}}
\newcommand{\R}{\operatorname{Re}}

\newcommand{\RR}{\mathbb{R}}

\newcommand{\CC}{\mathbb{C}}

\newcommand{\ep}{\varepsilon}

\newcommand{\ol}{\overline}

\newcommand{\calG}{\mathcal{G}}

\newcommand{\calI}{\mathcal{I}}

\definecolor{gr}{rgb}   {0.,   0.8,   0. }
\definecolor{bl}{rgb}   {0.,   0.5,   1. }
\definecolor{mg}{rgb}   {0.7,  0.,    0.7}

\def\Xint#1{\mathchoice
   {\XXint\displaystyle\textstyle{#1}}%
   {\XXint\textstyle\scriptstyle{#1}}%
   {\XXint\scriptstyle\scriptscriptstyle{#1}}%
   {\XXint\scriptscriptstyle\scriptscriptstyle{#1}}%
   \!\int}
\def\XXint#1#2#3{{\setbox0=\hbox{$#1{#2#3}{\int}$}
     \vcenter{\hbox{$#2#3$}}\kern-.5\wd0}}
\def\aver#1{\Xint-_{#1}}

\DeclareMathOperator*{\essinf}{ess\,inf}

\begin{document}

\allowdisplaybreaks

\title{Vertical versus conical square functions
}

\author{Pascal Auscher}

\address{Pascal Auscher
\\
Univ.~Paris-Sud, laboratoire de Math\'ematiques, UMR 8628, Orsay F-91405; CNRS, Orsay, F-91405}

\email{pascal.auscher@math.u-psud.fr}

\author{Steve Hofmann}

\address{Steve Hofmann
\\
Department of Mathematics,
University of Missouri, Columbia, Missouri 65211, USA}

\email{hofmann@math.missouri.edu}

\author{Jos\'e-Mar{\'\i}a Martell}

\address{Jos\'e Mar{\'\i}a Martell
\\
Instituto de Ciencias Matem\'aticas CSIC-UAM-UC3M-UCM
\\
Consejo Superior de Investigaciones Cient{\'\i}ficas
\\
C/ Nicol\'as Cabrera, 13-15
\\
E-28049 Madrid, Spain} 
\email{chema.martell@icmat.es}

\thanks{Part of this work was carried out while the first author was visiting the Centre for Mathematics and its Applications, Australian National University, Canberra ACT 0200, Australia. The second author was partially supported by NSF grant number DMS 0801079. The third author was supported by MEC Grant MTM2007-60952.}

\date{\today}

\begin{abstract}
We study the difference between vertical and conical square functions in the abstract  and also  in the specific case where the square functions come from an elliptic operator.
\end{abstract}



\maketitle

\section{Introduction}

The purpose of this article is to draw attention to differences between vertical and conical square functions. By vertical square functions, we mean the usual Littlewood-Paley-Stein functionals. By conical square functions, we mean the area functionals of Lusin type. Our interest in this subject was triggered by the recent  unpublished work of Dragi\v{c}evi\'c and Volberg \cite{DV}. Let us first describe what they proved.

Let $A=A(x)$ be an $n\times n$ matrix of complex, $L^\infty$
coefficients,  defined on $\mathbb{R}^n$, and satisfying the
\begin{equation*}
\lambda |\xi |^2\leq \R  A
\xi \cdot
\overline\xi
\ \textrm{and}\ | A \xi \cdot \overline \zeta
| \leq \Lambda |\xi ||\zeta|,\end{equation*} for $\xi,\zeta \in
\mathbb{C}^n$ and for some $\lambda ,\Lambda$ such that $0<\lambda \leq \Lambda
<\infty$.
 We define a second order divergence form operator
\begin{equation*}
Lf\equiv -\dv (A\nabla f),\end{equation*} which we interpret in the sense of maximal
accretive operators via a sesquilinear form.

\begin{proposition}[\cite{DV}]\label{propDV} 
If $A$ is real and $1<p<\infty$, there is a dimension free bilinear estimate
 \begin{equation}\label{eqDV}\iint_{\RR^{n+1}_+}  |\nabla_{y} e^{-tL}f(y)|  |\nabla_{y} e^{-tL}g(y)|dydt \le C(p, \lambda, \Lambda) \|f\|_p\|g\|_{p'}.
 \end{equation}
  Here $p'$ is the conjugate exponent to $p$. \end{proposition}

Set aside the dimension free bound, this  result is striking
in view of the following vertical square function estimate.

\begin{proposition}[\cite{A}]\label{propA} If $A$ is real and $1<p<q_{+}(L)$,
 \begin{equation}
\label{eq:verticalh}
\int_{\RR^{n}}\left(\int_{0}^\infty  |\nabla_{y} e^{-tL}f(y)|^2  dt\right)^{p/2}\, dy \le C(p, n, \lambda, \Lambda)^p \|f\|_p^p.
\end{equation}
 Furthermore, this estimate fails for $p>q_{+}(L)$ (if $q_{+}(L)<\infty$).
\end{proposition}

The number $q_{+}(L)$ has been introduced in \cite{A}, as well as the three other numbers  $p_{-}(L)$, $q_{-}(L)$ and $ p_{+}(L)$  as limits of the following intervals.
The interval   $(p_-(L),p_+(L))$ is the maximal open interval where the heat semigroup $\{e^{-tL}\}_{t>0}$ is uniformly bounded on $L^p$ or equivalently the semigroup satisfies $L^p-L^q$ off-diagonal estimates  when $p_{-}(L)<p\le q <p_{+}(L)$  , see \eqref{off-heat} below. Analogously, $(q_-(L),q_+(L))$ is the maximal open interval where $\{\sqrt{t}\nabla e^{-tL}\}_{t>0}$ is uniformly bounded on $L^p$ or satisfies $L^p-L^q$ off-diagonal estimates when $q_{-}(L)<p\le q <q_{+}(L)$  .  These intervals also determine up to endpoints the range of $L^p$ boundedness  of the functional calculus, Riesz transform and  vertical square functions. One has $p_-(L)=q_-(L)$, $(q_+(L))^*\le p_+(L)$ ---where $q^*=qn/(n-q)$ when $q<n$ and $q^*=\infty$ otherwise. Also, $p_-(L)=q_-(L)=1$ and $p_+(L)=q_{+}(L)=\infty$ if $n=1$.  For  $n=2$, or for $n\ge 3$ and  $L$ with real coefficients, the same is true except that one can only say that $q_+(L)>2$ for $n\ge 2$ and this is sharp. Additionally, if $n\ge 3$ and $A$ with complex coefficients, then $p_-(L)<2n/(n+2)$ and $p_+(L)>2n/(n-2)$.  See \cite{A} for full details. Also \cite{HMMc}   proves the latter inequalities to be sharp using an example of Freshe \cite{Fr}.

 Since $q_+(L)$ can be arbitrary close to 2, one cannot deduce the bilinear estimate in Proposition \ref{propDV} from the vertical square function estimate. So the bilinear estimate seems to exhibit some special feature that the vertical square function does not have. Indeed,   bilinear integrals as above can also be estimated using conical square functions thanks to an averaging trick that appears
in \cite{FS} and \cite{CMS}:
$$ \iint_{\RR^{n+1}_{+}}F(y,t) G(y,t) dydt  = b_{n}^{-1}\int_{\RR^n} \left( \iint_{\RR^{n+1}_{+}} F(y,t) G(y,t) h\Big(\frac{x-y}{t^{1/2}}\Big)\, \frac{dydt}{t^{n/2}}\right) dx
$$
with $h$ the indicator function of the unit ball and $b_{n}$ its volume, so that
$$
\left | \iint_{\RR^{n+1}_{+}} F(y,t) G(y,t) dydt \right| \le c_{n} \|S_{h}F\|_{p}\|S_{h}G\|_{p'}$$
with
$$
S_{h}F(x)= \left(\iint_{|x-y| <  \sqrt t}  |F(y,t)|^2 \frac
{dydt}{{t }^{n/2}}\right)^{1/2}.
$$
Hence, applying this to $F(y,t)=\nabla_{y} e^{-tL}f(y)$ and $G(y,t)=\nabla_{y} e^{-tL}g(y)$ it becomes natural to expect the corresponding conical square function estimate  holds in a larger range of $p$ than the one for \eqref{eq:verticalh}. Indeed, we shall show as part of  Theorem \ref{newprop}.
\begin{proposition}\label{prop:Conical-real}
If $A$ is real and $1<p<\infty$,
\begin{equation}
\label{eq:conicalh}
\int_{\RR^n} \left(\iint_{|x-y| <  \sqrt t}  |\nabla_{y} e^{-tL}f(y)|^2 \frac
{dydt}{{t }^{n/2}}\right)^{p/2} \,dx \le C(p, n,\lambda,\Lambda)^p\|f\|_{p}^p.
\end{equation}
\end{proposition}
 Thus \eqref{eqDV} holds  at least  with a  dimension dependent bound . We shall also study \eqref{eq:conicalh} for all complex $A$ and show it holds when $p_{-}(L)<p<\infty$ and fails when $p<p_{-}(L)$ (if $p_{-}(L)>1$). This is consistent as $p_{-}(L)=1$ when $A$ is real. This  also improves  \cite[Corollary 6.10]{A} where \eqref{eq:conicalh} was  obtained in the range
$p_{-}(L)<p<q_{+}(L)$. The  bilinear inequality as in  Proposition \ref{propDV}  then holds  for a restricted range $p_{-}(L)<p<p_{-}(L)'=p_{+}(L^*)$.

This leads us to the main point of this article about  comparison between vertical and conical square functions. Propositions \ref{propDV} and \ref{prop:Conical-real} show that the ranges of $p$ below 2 are the same but differ above 2. One may wonder whether there is an abstract principle behind this. But this is not the case.  Aside from $p=2$ for which the averaging trick yields that they are equivalent, vertical and conical square functions only compare one way  for $p\ne 2$ in the sense that one is automatically controlled by the other and simple examples show the converses fail. More precisely, for $p>2$, it is well-known and we shall recall why in Section \ref{sec:CvsV}, that a  vertical square function controls the corresponding conical one.
We shall also prove, and it seems this is not in the literature, that for $p<2$, the conical square function controls the corresponding vertical one. Comparing the ranges for \eqref{eq:conicalh} and \eqref{eq:verticalh} already furnishes
a counter-example for the converse in the $p>2$ range and an example where the converse holds in the $p<2$ range.  We note this can be done on a space of homogeneous type.
 We shall also study some weigthed comparisons using extrapolation.

We finish this introduction by the following observations. As explained before the range of $p$ for \eqref{eq:verticalh} is tight to the range of $L^p$ boundedness for $\sqrt{t} \nabla e^{-tL}$. As the $p<2$ range for \eqref{eq:conicalh} is \textit{a priori} smaller than or equal to the $p<2$ range for  \eqref{eq:verticalh},  we obtain the best possible result by showing they are equal. For $p>2$ we exhibit a new phenomenon.

Our results show that  the $p>2$ range for  \eqref{eq:conicalh} is
linked  to the rate of decay in the $L^2$ off-diagonal estimates. If the latter
is fast enough then one obtains the full range $(2,\infty)$ as it is the case in \eqref{eq:conicalh}.  In fact this $L^p$ estimate amounts  to proving  boundedness  of some vector-valued operator from $L^p$ into the parabolic version of the tent space $T^p_{2}$ of Coifman, Meyer, Stein \cite{CMS}.
When the $L^2$ off-diagonal decay is fast enough we can prove, basically following the Fefferman-Stein argument, that $|\nabla_{y} e^{-tL}f(y)|^2
{dydt}$ is a parabolic  Carleson measure for $f\in L^\infty$, which is nothing but an  $L^\infty \to T^\infty_{2}$ estimate. One can then interpolate for $2<p<\infty$.
When the rate of decay is slow (for example polynomial with small exponent) this argument does not seem to adapt and one needs other tools.  This is the case for the conical square function based on  $\varphi(t^2L)$ when $A$ is complex and $\varphi$ not smooth as the origin. An example is the Poisson semigroup  since $\varphi(z)=e^{-z^{1/2}}$ in this case. A different ingredient then comes into play, which is the decay at 0 of $\varphi$ or the order at which it vanishes, combined with the definition of $p_{+}(L)$. For instance, in Section \ref{ss3.4} below we shall prove the following  and this is the $p>2$ range that is interesting for our discussion here.  
\begin{proposition}\label{prop1.4} For  $m$  a non negative integer and   $f\in L^p$, then
\begin{equation} \label{eq:m}
\int_{\RR^n} \left(\iint_{|x-y| <  t}  |t\nabla_{y,t}\big((t^2L)^me^{-tL^{1/2}}f\big)(y)|^2 \frac
{dydt}{{t}^{n+1}}\right)^{p/2} \, dx \le C \|f\|_{p}^p
\end{equation}
whenever
\begin{equation}\label{range:m}
p_{-}(L)<p< \frac{np_{+}(L)}{n-(2m+1)p_{+}(L)}.
\end{equation}
For $(2m+1)p_{+}(L) \ge n$ then the right hand side in \eqref{range:m} becomes $\infty$.
\end{proposition}
 If $L$ had been the Laplacian, $2m+1$ would just be the number of vanishing moments for the kernel of the convolution operator $\nabla_{y,t}(t^2L)^m e^{-tL^{1/2}}$ or the number of derivatives in front of the semigroup. Here we have $2m$th order ``vanishing" coming from the exponent of the second order operator $L$, and +1 comes from the gradient.

A few comments are in order.  We first point out that,
since vertical and conical square functions are equivalent in  $L^2$,
one may integrate by parts in $t$ and use properties of the semigroup to pass from any choice
of non-negative integer $m$ to another, in the case $p=2$.  In fact, one may even take $m$ of the form
$m=k+1/2$, with $k$ a non-negative integer.  For $p>2$, different values of $m$ apparently need no longer be equivalent;  the conclusion of the proposition yields a better range of $p$ for larger $m$ (up to the critical
value with $(2m+1)p_+(L)=n$).

In particular,  the case $m=0$  of \eqref{eq:m} gives standard
area integral estimates for weak solutions of the equation
\begin{equation}
\label{eq:elli}
\partial_t^2 u + \dv A \nabla u=0.
\end{equation}
When $A$ is real (in which case $p_-(L)=1,\,p_{+}(L)=\infty)$), such estimates may be obtained
as a special case (the ``block matrix" case), of the result of
Dahlberg, Jerison, Kenig in  \cite{DJK}, using the fact that one has non-tangential estimates
for the solutions $u(\cdot,t):= e^{-t\sqrt{L}}f$ in every $L^p, 1<p\leq \infty.$
The present argument allows for a direct (and simpler) proof
than that in \cite{DJK} in this special case.  Moreover, it
has the added virtue of applying to the case of complex coefficients.  Of course, we do not address the question of ``full" coefficient matrices (i.e., those that need not be in block form), as is done in \cite{DJK}.

\subsection*{Acknowledgments} This work was started years ago while the authors were all visiting the Universidad Aut\'onoma in Madrid on the occasion of a special program in harmonic analysis and, after a latency period, finished this year while the authors were all visiting the Center for Mathematics and Applications of the Australian National University. We are very grateful to these institutions for hospitality and financial support.
Also we want to express our  thanks to O. Dragi\v{c}evi\'c and A. Volberg who showed us their unpublished work.

\section{Vertical versus conical}\label{sec:CvsV}

For a locally square integrable  function $f$ on $\RR^{n+1}_{+}$, denote $$
Sf(x)= \left(\iint_{|x-y| <  t}  |f(y,t)|^2 \,\frac
{dydt}{{t}^{n+1}}\right)^{1/2}, \quad x\in \RR^n.
$$
and
$$
Vf(x)= \left(\int_{t>0}  |f(x,t)|^2 \, \frac
{dt}{{t}}\right)^{1/2}, \quad x\in \RR^n.
$$
$Sf$ is lower semi-continuous hence a measurable function. Measurability on $Vf$ follows from the local square integrability of $f$.

We remark that
$$
\|Sf\|_{2}^2= b_{n} \|Vf\|_{2}^2
$$
with $b_{n}$ the volume of the unit Euclidean ball.

\subsection{Comparison in Lebesgue spaces}

\begin{proposition}\label{prop:comp} Let $f$ be locally square integrable on $\RR^{n+1}_{+}$.
\begin{list}{$(\theenumi)$}{\usecounter{enumi}\leftmargin=1cm \labelwidth=1cm \itemsep=0.1cm \topsep=.2cm \renewcommand{\theenumi}{\alph{enumi}}}
  \item For $2<p<\infty$, $$
\|Sf\|_{p} \le C(p,n)\|Vf\|_{p}.
$$
  \item For $0<p<2$, $$
\|Vf\|_{p} \le C(p,n)\|Sf\|_{p}.
$$
\item The converses fail for all $p\ne 2$.
\end{list}
\end{proposition}

\begin{proof}
Part $(a)$ is standard and appears already in \cite[p. 91]{St}.  For the sake of self-containment, we
recall the argument. As $p>2$, $q=p/2>1$ and we can estimate $\|Sf\|_{p}=\|(Sf)^2\|_{p/2}^2$ by dualizing against a function $h\in L^{q'}$. Now, the averaging trick applies and yields
\begin{align*}
   \int_{\RR^n} (Sf)^2(x) h(x) \, dx &= b_{n}\iint_{\RR^{n+1}_{+}} |f(y,t)|^2 \left(\frac1{|B(y,t)|}\,\int_{B(y,t)} h(x)\,dx\right)
 \,\frac{dydt}{{t}}
    \\
    &
 \le b_{n}\int_{\RR^n} (Vf)^2(y) Mh(y)\, dy
 \\
    &
 \le
 b_{n}  \|Vf\|_{p}^2\|Mh\|_{q'}
\end{align*}
and one concludes using the boundedness of the  maximal operator over balls $M$  in $L^{q'}$.

Let us now prove Part $(b)$.  Fix $0<p<2$, $f$ with $Sf\in L^p$ and $\lambda>0$. Since $x\mapsto Sf(x)$ is  lower semi-continuous, the  set $O=\{Sf>\lambda\}$  is open.
 Let $F$ be the complement of $O$ in $\RR^n$,
 $R(F)$ be the union of the cones $|x-y|<t$ with vertices $x\in F$. We also set
$\tilde{O}=\{x\in\RR: M(\chi_O)(x)>1/2\}$ and $\tilde{F}=\RR^n\setminus\tilde{O}$.
We note that $O\subset\tilde{O}$ since $O$ is open, and thus $\tilde{F}\subset F$. If $y\in \tilde F$ and $t>0$ we have $|O\cap B(y,t)|/|B(y,t)|\le 1/2$ and consequently, $|F\cap B(y,t)|/t^n\ge b_n/2$.
Hence,
\begin{align*}
 \int_{F} (Sf)^2(x)\, dx
& =
 \int_{x\in F} \iint_{|x-y| <  t}  |f(y,t)|^2\,  \frac
{dxdydt}{{t}^{n+1}}
\\
&
= \iint_{R(F)}  \frac{|F\cap B(y,t)|}{t^n}|f(y,t)|^2 \,\frac{dydt}{{t}}.
\\
&
\ge \frac{b_{n}}{2} \int_{y\in \tilde F} \int_{t>0}  |f(y,t)|^2\,  \frac{dydt}{{t}}
\\
&
=
\frac{b_{n}}{2} \int_{\tilde F} (Vf)^2(y)\, dy
\\
&
 \ge
\frac{b_{n}}{2} \lambda^2 |\{Vf>\lambda\} \cap \tilde F|.
\end{align*}
Besides, for $0<r<p$, using the weak type (1,1) for $M$,
$$
|\{Vf>\lambda\} \cap\tilde O|\le |\tilde O| \le 2\cdot 3^n\,|O| \le  \frac{ 2\cdot 3^n}{\lambda^r}\int_{O}(Sf)^r(x)\, dx.$$
Hence,
\begin{align*}
\int_{\RR^n} (Vf)^p(x)\, dx
    & = p\int_{0}^\infty \lambda^{p-1} |\{Vf>\lambda\}|\, d\lambda  \\
    &  \le \frac{2p}{b_{n}} \int_{0}^\infty \lambda^{p-2-1} \int_{Sf\le \lambda} (Sf)^2(x)\, dxd\lambda   \\
    &\qquad\qquad +   {2p \cdot 3^n}{} \int_{0}^\infty \lambda^{p-r-1} \int_{Sf> \lambda} (Sf)^r(x)\, dxd\lambda
    \\
    &= \left( \frac{2p}{b_{n}(2-p)}
  + \frac{2p\cdot 3^n}{p-r}\right) \int_{\RR^n} (Sf)^p(x)\, dx.
\end{align*}

We now finish the proof with Part $(c)$. It is convenient to introduce $$
\tilde{S}f(x)= \iint_{|x-y| <  t}  |f(y,t)| \frac
{dydt}{{t}^{n+1}},
\qquad
\tilde{V}f(x)= \int_{t>0}  |f(x,t)| \frac
{dt}{{t}}.
$$
Note that $S f=\tilde{S}(|f|^2)^{1/2}$ and $V f=\tilde{V}(|f|^2)^{1/2}$, so that,
for a locally integrable function $f$ on $\RR^{n+1}_{+}$, we seek to disprove the inequalities
\begin{align}
\|\tilde{S} f\|_{p}&\le C\,\|\tilde{V} f\|_{p}, \quad 0<p<1; \label{Sf-Vf-counter}
\\[0.2cm]
\|\tilde{V} f\|_{p}&\le C\,\|\tilde{S} f\|_{p}, \quad 1<p<\infty. \label{Vf-Sf-counter}
\end{align}
We write the argument so that it is easy to adapt it to a space of homogeneous type, denoting $v(B)$ the volume of a ball and using implicitly the doubling property in the argument. See Remark \ref{remark:SHT} below.

For \eqref{Sf-Vf-counter} we consider $f_N(x,t)=N^{-1}\,t\,\chi(x)\,\chi_0(t/N)$ with $N\gg 1$ and where $\chi$ is the characteristic function of the unit ball $B(0,1)$ and $\chi_0$ denotes the characteristic function of the interval $[0,1]$. On the one hand,
$$
\tilde{V} f_N(x)=
\int_0^\infty |f_N(x,t)| \frac{dt}{t}
=
N^{-1}\,\chi(x)\,\int_0^{N}\,dt
=
\chi(x)
$$
and therefore
$$
\|\tilde{V} f_N\|_{p}^p
= v(B(0,1)).
$$
On the other hand, fixed $|x|\le N/8$, if $|y|\le 1$ we have $|x-y|<N/4$ (provided $N>8$) and then 
\begin{align*}
\tilde{S} f_N(x)
&=
N^{-1}\,\iint_{|x-y| <  t}  t\,\chi(y)\,\chi_0(t/N)\frac
{dydt}{{t}v(B(y,t))}
\\
&
=
N^{-1}\,\int_{|y|\le 1}\int_{|x-y|<t\le N} \frac{dtdy}{v(B(x,t))}
\\
&\ge
N^{-1}\,\int_{|y|\le 1}\int_{N/4<t\le N} \frac{dtdy}{v(B(x,t))}
\\
&\ge C\,  \frac{v(B(0,1))}{v(B(0,N))}.
\end{align*}
This implies
$$
\|\tilde{S} f_N\|_{p}^p
\ge
\int_{|x|\le N/8} \tilde{S} f(x)^p\,dx
\ge C\, \frac{v(B(0,1))^p}{v(B(0,N))^{p-1}}.
$$
Gathering the obtained estimates
$$
\frac{\|\tilde{S} f_N\|_{p}^p}{\|\tilde{V} f_N\|_{p}^p} \ge C\, \frac{v(B(0,N))^{1-p}}{v(B(0,1))^{1-p}}.$$
Thus \eqref{Sf-Vf-counter} cannot hold as $v(B(0,N))$ increases to $\infty$ and $1-p>0$.

 For \eqref{Vf-Sf-counter} we consider $f_N(x,t)= \,t\,v(B(x,t))\, \chi_{N}(x)\,\chi_0(t)$ with $N\gg 1$ where $\chi_{N}$ is the characteristic function of the ball of radius $1/N$. We first calculate $\tilde{V} f_N$:
\begin{align*}
\tilde{V} f_N(x)
&=
\int_0^\infty |f_N(x,t)| \frac{dt}{t}
\\
&\ge
\,\chi_{N}(x)\,
\int_{1/2}^{1}\,v(B(x,t))\,{dt}
\\
&\ge
C \,\chi_{N}(x) v(B(x,1))
\\
&\ge C \,\chi_{N}(x) v(B(0,1))
\end{align*}
and therefore
$$
\|\tilde{V} f_N\|_{p}^p
\ge C v(B(0,1/N)) v(B(0,1))^p.
$$
We find an upper bound for  $\|\tilde{S} f_N\|_{p}$. We notice that if $|x|>2$, $|y|\le 1/N$ and $0\le t\le 1$ we have
 $|x-y|>1\ge t$ (if $ N\ge 1$). Thus, $\tilde{S}f_{N}(x)=0$ if $|x|>2$. On the other hand, for all $x\in \RR^n$:
\begin{equation*}
\tilde{S} f_N(x)
=
\,\iint_{|x-y| <  t\le 1}  \,\chi_{N}(y)\,
{dydt}{}
\\
\le
v(B(0,1/N)).
\end{equation*}
Then, we obtain
$$
\|\tilde{S} f_N\|_{p}^p
=
\int_{|x|\le 2} \tilde{S} f_N(x)^p\,dx
\le C v(B(0,1/N))^p v(B(0,1))
$$
so that
$$
\frac{\|\tilde{S} f_N\|_{p}^p}{\|\tilde{V} f_N\|_{p}^p} \le C\, \frac{v(B(0,1/N))^{p-1}}{v(B(0,1))^{p-1}}$$
which goes to 0 as $N\to \infty$ if $p>1$.
\end{proof}

\begin{remark}\label{remark:SHT} The reader can notice that this theorem generalizes to spaces of homogeneous type $X$ with infinite volume and at least one point that is not an atom (which plays the role of $0$). That is $\RR^{n+1}_{+}$ is changed to $X\times \RR_{+}$ and in the definition of $Sf(x)$ the measure has to change to $\frac{d\mu(y)dt}{t\mu(B(y,t))}$.
\end{remark}

\subsection{Weighted estimates via extrapolation}

Let us present a  weighted version of  Proposition \ref{prop:comp} using extrapolation.  That is, $L^2$ estimate with suitable   Muckenhoupt weights imply  $L^p$ comparisons in weighted spaces.   Let $A_{p}$, $1\le p<\infty$, denote the classical Muckenhoupt classes of weights and $RH_{p}$, $1<p\le \infty$, the class of reverse H\"older weights. See for example \cite{AM1}. Again everything extends to a space of homogeneous type as in the remark above. We stick to the Euclidean space for simplicity.

\begin{proposition}\label{prop:weightedcomp} Let $f$ be a locally square integrable function on $\RR^{n+1}_{+}$.

\begin{list}{$(\theenumi)$}{\usecounter{enumi}\leftmargin=1cm \labelwidth=1cm \itemsep=0.1cm \topsep=.2cm \renewcommand{\theenumi}{\alph{enumi}}}

\item For $2<p<\infty$ and $w\in A_{p/2}$
$$
\|S f\|_{L^p(w)}
\le
C(p,w)\,
\|V f\|_{L^p(w)}.
$$
  \item For $0<p<2$ and $w\in RH_{(2/p)'}$
$$
\|V f\|_{L^p(w)}
\le
C(p,w)\,
\|S f\|_{L^p(w)}.
$$
\end{list}
\end{proposition}

\begin{proof} We begin with Part $(a)$.  Given any $w\in A_\infty$  we easily have
\begin{align}
\label{Sf-w}
\|S f\|_{L^2(w)}^2
&=
\int_{\RR^n}\,\iint_{|x-y| <  t}  |f(y,t)|^2 \frac
{dydt}{{t}^{n+1}}\,w(x)\,dx
\\
&=
b_n\int_{\RR^n}\,\int_0^\infty  |f(y,t)|^2 \frac{w(B(y,t))}{|B(y,t)|}
\frac
{dydt}{t}.\nonumber
\end{align}

We note that if $w\in A_1$, that is,  $Mw(y)\le [w]_{A_1}\,w(y)$ for a.e. $y\in \RR^n$, then  we have for all $t>0$
$$
\frac{w(B(y,t))}{|B(y,t)|}
\le
Mw(y)\le
[w]_{A_1}\,w(y),\qquad \mbox{a.e. }y\in\RR^n.
$$
Then, we obtain
$$
\|S f\|_{L^2(w)}^2
\le
b_n\,[w]_{A_1}\,\int_{\RR^n}\,\int_0^\infty  |f(y,t)|^2
\frac{dt}{t}\,w(y)\,dy
=
b_n\,[w]_{A_1}\,\|Vf\|_{L^2(w)}^2.
$$
Next we invoke the Rubio de Francia extrapolation theorem (see \cite{Rub}, \cite{Gar} for the original result, and \cite{CMP1}, \cite{CMP2} for a statement written in terms of pairs of functions) for the pairs $\big((Sf)^2, (Vf)^2\big)$: the starting estimate in $L^1(w)$ for every $w\in A_1$ implies that for every $2<p<\infty$ and $w\in A_{p/2}$
$$
\|S f\|_{L^p(w)}
\le
C(p,w)\,
\|V f\|_{L^p(w)}.
$$
Strictly speaking, the argument applies whenever the left hand side is finite. This is the case if $f$ is \textit{a priori} bounded with compact support in $\mathbb{R}^{n+1}_+$.
Monotone convergence implies that the inequality is valid for all locally square integrable function $f$.

For the reverse estimate in Part $(b)$, we recall that $w\in RH_\infty$ if for every ball $B$ we have
$$
w(x)\le [w]_{RH_\infty}\,\frac1{|B|}\,\int_B w(y)\,dy,\qquad\mbox{a.e. }x\in B.
$$
Then, using Lebesgue's differentiation theorem we obtain that for a.e $y\in\RR^n$ and for all $t>0$
\begin{align*}
w(y)
&
\le
\sup_{0<\tau\le t} \frac1{|B(y,\tau)|}\,\int_{B(y,\tau)} w(x)\,dx
\\
&
\le
[w]_{RH_\infty}\,
\sup_{0<\tau\le t} \frac1{|B(y,\tau)|}\,\int_{B(y,\tau)} \frac1{|B(y,t)|}\,\int_{B(y,t)} w(z)\,dz \,dx
\\
&
=
[w]_{RH_\infty}\,\frac{w(B(y,t))}{|B(y,t)|}.
\end{align*}
Thus, for every $w\in RH_\infty$ by \eqref{Sf-w} we have
\begin{align*}
\|V f\|_{L^2(w)}^2
&=
\int_{\RR^n}\,\int_0^\infty  |f(y,t)|^2
\frac{dt}{t}\,w(y)\,dy
\\
&
\le
[w]_{RH_\infty}\,
\int_{\RR^n}\,\int_0^\infty  |f(y,t)|^2 \frac{w(B(y,t))}{|B(y,t)|}
\frac
{dydt}{t}
\\
&
=
[w]_{RH_\infty}\,b_n^{-1}\,\|S f\|_{L^2(w)}^2.
\end{align*}
Considering the pairs $(F,G)=\big((Vf)^2, (Sf)^2\big)$ we have obtained that
$$
\int_{\RR^n} F(x)\,w(x)\,dx\le [w]_{RH_\infty}\,b_n^{-1}\,\int_{\RR^n} G(x)\,w(x)\,dx,\qquad \forall\,w\in RH_\infty.
$$
We take an arbitrary $p_0$ with $0<p_0<1$ and set $q_0=r=1$. Then the last estimate holds in particular for every $w\in A_{r/p_0}\cap RH_{(q_0/r)'}$. We apply the extrapolation theorem for limited ranges \cite[Theorem 4.9]{AM1} (see also \cite{CMP2}) to conclude that for all $p_0<q<q_0$
$$
\int_{\RR^n} F(x)^q\,w(x)\,dx\le C_w\,\int_{\RR^n} G(x)^q\,w(x)\,dx,\qquad \forall\,w\in A_{q/p_0}\cap RH_{(q_0/q)'},
$$
whenever the left hand side is finite. This is the case when $f$ is bounded with compact support in $\mathbb{R}^{n+1}_+$ and can be removed by monotone convergence to allow all locally square integrable function $f$.
Next, we fix $0<q<1$ and $w\in RH_{(1/q)'}$. Then, $w\in A_\infty$ and there exists $0<p_0<q$ such that $w\in A_{q/p_0}$. Thus we can apply the last estimate since $0<p_0<q<1=q_0$ and $w\in A_{q/p_0}\cap RH_{(q_0/q)'}$. Hence we have proved that for every $0<p<2$ and $w\in RH_{(2/p)'}$
$$
\|V f\|_{L^p(w)}
\le
C(p,w)\,
\|S f\|_{L^p(w)}.
$$
\end{proof}

Notice that from the argument one sees that the extrapolations take initial estimates in  $L^1(w)$. Indeed, from the beginning one could have worked with the operators $\tilde{S}$ and $\tilde{V}$ defined above.
The argument just presented shows that for every locally integrable function $f$ on $\RR^{n+1}_{+}$, if $1\le p<\infty$ and $w\in A_{p}$ then
$$
\|\tilde{S} f\|_{L^p(w)}
\le
C(p,w)\,
\|\tilde{V} f\|_{L^p(w)},
$$
and if $0<p\le 1$ and $w\in RH_{(1/p)'}$, then
$$
\|\tilde{V} f\|_{L^p(w)}
\le
C(p,w)\,
\|\tilde{S} f\|_{L^p(w)}.
$$

\section{Square functions for typical functions of $L$}

Consider the operator $L$ defined in the Introduction. We introduce the following conical and vertical square functions
$$
{\calG}_{P}(f)(x)=\left(\iint_{|x-y| <  t}  |t\nabla_{y,t}e^{-tL^{1/2}}f(y)|^2 \frac
{dydt}{{t}^{n+1}}\right)^{1/2},
$$
$$
{G}_{P}(f)(y)=\left(\int_{t>0}  |t\nabla_{y,t}e^{-tL^{1/2}}f(y)|^2 \frac
{dt}{{t}{}}\right)^{1/2},
$$
$$
{\calG}_{h}(f)(x)=\left(\iint_{|x-y| <  \sqrt t}  |\nabla_{y}e^{-tL}f(y)|^2 \frac
{dydt}{{t }^{n/2}}\right)^{1/2},
$$
$$
{G}_{h}(f)(y)=\left(\int_{ t>0}  |\nabla_{y}e^{-tL}f(y)|^2
{dt}{{}{}}\right)^{1/2},
$$

The $P$ subscript refers to the fact that we are dealing with the Poisson semigroup $e^{-tL^{1/2}}$ for $L$. The $h$ subscript refers to the heat semigroup $e^{-tL}$. The curly letters are for the conical square functions and the capital letters for the vertical ones. So from our general observations we know that $\|{\calG}_{P}(f)\|_{p} \lesssim \|{G}_{P}(f)\|_{p}$ for $2\le p <\infty$ and
$\|{G}_{P}(f)\|_{p} \lesssim \|{\calG}_{P}(f)\|_{p} $ for $0<p\le 2$ and similarly for the heat versions by making a change of variables in $t$.
Note that these square functions all contain a spatial gradient. Hence we are not working within the functional calculus of $L$.

We want to compare the $L^p$ norms of each square functions with the $L^p$ norm of the original function $f$.

For $p=2$, a mere integration  by parts  (see \cite{A})  yields that
$$
\|G_{P}(f)\|_2 + \|G_{h}(f)\|_2 \approx C(\lambda, \Lambda) \|f\|_2.
$$
As seen above, conical square functions behave as the vertical ones in $L^2$.

We turn to a summary of results on $L^p$. Let $p^*=\frac{np}{n-p}$ if $p<n$ and $\infty$ otherwise.
Let us remind the reader that the exponents $p_\pm(L),\,q_\pm(L)$ were defined in the introduction,
in the discussion following the statement of Proposition \ref{propA}.

\begin{theorem}\label{newprop}
\begin{enumerate}
 \item ${G}_{h}$ is bounded on $L^p$ for $p_{-}(L)<p<q_{+}(L)$.
  \item ${\calG}_{h}$ is bounded on $L^p$ for $p_{-}(L)<p<\infty$.
 \item ${G}_{P}$ is bounded on $L^p$ for $p_{-}(L)<p<q_{+}(L)$.
   \item ${\calG}_{P}$ is bounded on $L^p$ for $p_{-}(L)<p<p_{+}(L)^*$.
\end{enumerate}

The upper bounds are optimal except maybe for ${\calG}_{P}$.  The lower bounds are all optimal.

The converse  estimates $\|f\|_{p}\lesssim \|g(f)\|_{p}$ are valid for all $p\in (1,\infty)$ and $f\in L^p\cap L^2$ and all four square functions. Hence, each defines a  new norm on $L^p$ for $p$ in the corresponding  range above.
\end{theorem}

Fix $\mu\in (0,\pi/2)$ and $1\le p\le \infty$. We say, following \cite{A}, that  a family of linear operators $(T_{z})_{z\in \Sigma_{\mu}}$ satisfies $L^p-L^q$ off-diagonal estimates if there exist constants $c,C$ such that for all $z\in \Sigma_{\mu}:=\{z\in \CC^*; |\arg z|<\mu\}$,  all Borel sets $E$, $F$ and all $f\in L^p(E)$, we have
\begin{equation}\label{off-heat}
\|T_{z} (f\,\chi_E)\|_{L^q(F)}
\le
C\, {|z|^{-\frac n 2\,(\frac1p-\frac1q)}} e^{- \tfrac{cd(E,F)^2}{|z|}}\,\|f\|_{L^p(E)}.
\end{equation}

This holds for $\mu<\pi/2-\omega$ with $\omega$ the type of $L$, $T_{z}= (zL)^me^{-zL}$ with  $p_-(L)<p\le q<p_+(L)$  and
$T_{z}=|z|^{1/2}\nabla(zL)^me^{-zL}$ with $q_{-}(L)=p_-(L)<p\le q <q_{+}(L)$, for any non-negative integer $m $. See \cite[Chapter 3]{A}.

For $1\le p <\infty$, we recall that the tent space $T^p_{2}$ denotes the space of locally square integrable functions in $\RR^{n+1}_{+}$ such that $Sf \in L^p(\RR^n)$ with the notation of Section \ref{sec:CvsV}. The norm in $T^p_{2}$ is given by $\|Sf\|_{p}$ as defined in Section \ref{sec:CvsV}. Note that changing the aperture of cones yields equivalent norms. For $p=\infty$, we let $T^\infty_{2}$ be the space of  locally square integrable functions in $\RR^{n+1}_{+}$ such that
$$
\|f\|_{T^\infty_{2}}=\sup_{B}\left( \frac {1}{|B|}\iint_{B\times (0,r_{B})}  |f(y,t)|^2 \,\frac
{dydt}{{t}}\right)^{1/2}<\infty,
$$
the supremum being taken above all balls and  $r_{B}$ denotes  the radius of $B$.
The spaces $T^p_{2}$, $1\le p\le \infty$, form a complex interpolation family.
For more see \cite{CMS}.
Note that the $L^p$ boundedness of a conical square function reformulates canonically as an
$L^p$ to $T^p_{2}$ boundedness.

We first prove boundedness and sharpness for each square function. We consider next the converse inequalities globally.

\subsection{Proof of Theorem   \ref{newprop} for ${G}_{h}$} This  was treated in \cite{A}. There the range of $p$  is shown to be the largest possible open set.

\subsection{Proof of Theorem  \ref{newprop} for ${\calG}_{h}$} For $p\le 2$, it is in \cite{A}. For
$p=\infty$, we first obtain the boundedness of $f \mapsto t\nabla e^{-t^2L}f$ from $L^\infty$ to $T^\infty_{2}$ by a well-known argument of Fefferman-Stein \cite{FS}.
More precisely, we fix a ball $B$ and  write $f=f_{\rm loc}+f_{\rm glob}$ where $f_{\rm loc}=f\,\chi_{4\,B}$. Using the $L^2$ boundedness of ${\calG}_{h}$,
$$
\frac1{|B|}\,\iint_{\hat B} |t\nabla e^{-t^2L} f_{\rm loc}(x)|^2\,\frac{dx\,dt}{t}
\lesssim
\frac1{|B|}\|{G}_{h}f_{\rm loc}\|_{2}^2\\
\lesssim
\,\frac1{|B|}\|f_{\rm loc}\|_{2}^2\
\lesssim \|f\|_{\infty}^2.
$$
Next, the off-diagonal decay \eqref{off-heat} with $p=q=2$ for  $t^{1/2}\nabla e^{-tL}$ implies for some $0<c, C<\infty$,
$$
\frac1{|B|}\,\int_{B} |t\nabla e^{-t^2L} f_{\rm glob}(x)|^2\, {dx}
\le
 C \sum_{j\ge 2}e^{-\frac{c4^j r_B^{2}}{t^{2}}}\,\aver{2^{j+1}\,B}|f(x)|^2\,dx
 $$
which, integrated against $dt/t$ in $t\in (0, r_{B})$, yields a bound by $\|f\|_{\infty}^2.$

Then interpolate this estimate  with the boundedness from $L^2$ to $T^2_{2}$, to get boundedness from $L^p$ to $T^p_{2}$, which is the same as the $L^p$ boundedness of ${\calG}_{h}$ by rescaling $t^2\mapsto t$ in the integrals.

Note that compared to \cite{A}, the upper bound improves from $q_{+}(L)$ to $\infty$ and is of course optimal. As for the lower bounds, we have $\|G_{h}f\|_{p}\lesssim \|\calG_{h}f\|_{p}$ when $p\le 2$. Hence the fact that $p_{-}(L)$ is optimal for $G_{h}$ (see \cite{A}) implies the same for $\calG_{h}$.

\subsection{Proof of Theorem   \ref{newprop} for ${G}_{P}$} We begin with removing the $\nabla $ part in $G_{P}$ when $q_{-}(L)<p<q_{+}(L)$. We know that $\nabla L^{-1/2}$ is bounded on $L^p$ for $p$ in this range \cite{A}. So by vector-valued (in the Hilbert space $H=L^2(\RR^+, dt/t)$) extension (\cite[Proposition 4.5.9]{Gra}), we have that the
$$
\|G_{P}f\|_{p} \le C \| g_{P}f\|_{p}, \quad q_{-}(L) <p<q_{+}(L)
$$
with
$$g_P(f)(x)=\left(\int_{0}^\infty  |tL^{1/2}e^{-tL^{1/2}}f(x)|^2 \frac
{dt}{{t}}\right)^{1/2}.
$$
Next, \cite[Lemma 7.2]{HM}  using the subordination formula
\begin{equation}\label{subordination}
e^{-t\,L^{1/2}} f
=
C\,\int_0^\infty \frac{e^{-s}}{\sqrt{s}}\, e^{-\frac{t^2\,L}{4\,s}}f\,ds,
\end{equation}
 proves the pointwise  inequality $g_{P} \le C \tilde g_{h}$ with
$$\tilde g_h(f)(x)=\left(\int_{0}^\infty  |t^2Le^{-t^2L}f(x)|^2 \frac
{dt}{{t}}\right)^{1/2}.
$$
 and the latter is bounded on
$L^p$ for $p_{-}(L) <p<p_{+}(L)$.  This can be proved by adapting line by line  \cite[Theorem 6.1]{A}. This also follows from Le Merdy's theorem \cite[Theorem 3]{LeM}.

We conclude by noticing that $p_{-}(L)=q_{-}(L), q_{+}(L) < p_{+}(L)$ (when $q_{+}(L)<\infty$). This finishes the proof.

That the bounds $p_{-}(L)$ and $q_{+}(L)$ are sharp  follows by the same argument as for Step 7 in \cite[Theorem 6.1]{A}.

\subsection{Proof of Theorem  \ref{newprop} for ${\calG}_{P}$}\label{ss3.4}
More generally, we shall discuss here the proof of Proposition \ref{prop1.4}, for which Theorem
\ref{newprop} (4) represents the case $m=0$.  We shall treat the case $m>0$ explicitly only
when $p>2$,  as the cases $m=0$ and $m>0$ may be treated by the same argument
when $p\leq 2$.

This part of the proof of Theorem \ref{newprop} (and more generally, the proof of
Proposition \ref{prop1.4}) is the most involved as it does not follow from other known arguments in a simple way.  For $p=2$, this is classical integration by parts. We then present arguments for $p_{-}(L)<p<2$ and $2<p<p_{+}(L)^*$.
That $p_{-}(L)$ is sharp follows by the same argument as for $\calG_{h}$.

\subsubsection{$p< 2$} We present two proofs. The first one uses the recent theory of Hardy spaces adapted to $L$ from \cite{HM} or from \cite{BZ} and the second one adapts arguments in \cite{A} to prove instead weak-type bounds.  We remark as above
that the same proofs apply to prove \eqref{eq:m}. We shall omit details and stick to $m=0$.

\smallskip

\begin{proof}[Proof 1]
Consider the Hardy spaces $H^p_{L}$  defined in \cite{HM} for $p=1$ and \cite{HMMc}  for $p\ge 1$.  The  $ H^1_{L} \to L^1$ boundedness of ${\calG}_{P}$ is exactly  \cite[Theorem 5.6]{HM}. Then interpolation (see \cite[Lemma 4.24]{HMMc}) with the $p=2$ case, shows $ H^p_{L} \to L^p$ boundedness of ${\calG}_{P}$. Finally,  identification of $H^p_{L}$ with $L^p$ if and only if $p_{-}(L)<p<p_{+}(L)$ proved in \cite[Proposition 9.1]{HMMc} concludes the argument.

We mention that one can also use the abstract  Hardy spaces developed by Bernicot and Zhao in \cite{BZ} and the interpolation further developed in \cite{B1}. Namely   it suffices to prove an $L^1$ estimate on some abstract atoms (that is an $H^1_{F,ato}$ to $L^1$ estimate with $H^1_{F,ato}$ as in Section 3.3 of \cite{BZ}) and then interpolate.  By checking details and values (left to readers) from the clear presentation in \cite{B2}, one  exactly finds the range  for $L^p$ boundedness when $p<2$. This theory, compared to the \cite{HM} theory, has the advantage of not caring much about the ``right'' definition of the Hardy spaces as this is not needed for the purpose of interpolation.
\end{proof}

\begin{proof}[Proof 2]
We proceed as in \cite[p. 61]{A}. We need to adapt the proof of \cite[Theorem 1.1]{A} to the present situation. We take $A_r=I-(I-e^{-r^2\,L})^N$ with $N\ge 1$  an integer to be chosen and follow the proof of that result with $T=\calG_{P}$ and $p_-(L)<p<2$. As $\calG_{P}$ is bounded on $L^2$ and $A_r$ satisfies off-diagonal estimates in the range $(p_-(L),2]$ it suffices to show that
\begin{equation}\label{eqn:term-CZ}
I=\Big|\Big\{x\in\RR^n\setminus \cup_i 4\,Q_i: \calG_P\Big(\sum_i h_i\Big)(x)>\alpha/3\Big\}\Big|
\le
\frac{C}{\alpha^p}\,\int_{\RR^n} |f(x)|^p\,dx
\end{equation}
where $h_i=(I-A_{r_i})b_i$ and $r_i$ is the sidelength of the cube $Q_i$ given by the Calder\'on-Zygmund lemma \cite[Lemma 1.3]{A}.
We use Chebichev and Fubini
\begin{align*}
I
&\le
\frac{9}{\alpha^2}\,\int_{\RR^n\setminus \cup_i 4\,Q_i} \calG_P\Big(\sum_i h_i\Big)(x)^2\,dx
\\
&=
\frac{9}{\alpha^2}\,\int_{\RR^n\setminus \cup_i 4\,Q_i}
\iint_{|x-y|<t} \Big(\sum_i |t\,\nabla_{y,t} e^{-t\,L^{1/2}} h_i(y)|\Big)^2\,\frac{dy\,dt}{t^{n+1}}
\,dx
\\
&=
\frac{9}{\alpha^2}\,\iint_{\RR^{n+1}_+} \Big(\sum_i |t\,\nabla_{y,t} e^{-t\,L^{1/2}} h_i(y)|\Big)^2 \,\frac{|B(y,t)\setminus \cup_i 4\,Q_i|}{t^n}\frac{dy\,dt}{t}
\\
&\lesssim
\frac{1}{\alpha^2}\,\iint_{\RR^{n+1}_+} \Big(\sum_i \chi_{2\,Q_i}(y)\,|t\,\nabla_{y,t} e^{-t\,L^{1/2}} h_i(y)|\Big)^2 \,\frac{|B(y,t)\setminus \cup_i 4\,Q_i|}{t^n}\frac{dy\,dt}{t}
\\
&
\
+
\frac{1}{\alpha^2}\,\iint_{\RR^{n+1}_+} \Big(\sum_i \chi_{\RR^n\setminus 2\,Q_i}(y)\,|t\,\nabla_{y,t} e^{-t\,L^{1/2}} h_i(y)|\Big)^2 \,\frac{|B(y,t)\setminus \cup_i 4\,Q_i|}{t^n}\frac{dy\,dt}{t}
\\
&
=\frac{1}{\alpha^2} (I_{\rm loc}+I_{\rm glob}).
\end{align*}

We estimate $I_{\rm loc}$. Notice that since $y\in 2\,Q_i$ we have that $B(y,t)\subset 4\,Q_i$ for $t\le c\,r_i$. Then, using that the collection $\{2\,Q_i\}_i$ has bounded overlapping we obtain
\begin{align*}
I_{\rm loc}
&
\lesssim
\int_{\RR^n} \int_{c\,r_i}^{\infty} \Big(\sum_i \chi_{2\,Q_i}(y)\,|t\,\nabla_{y,t} e^{-t\,L^{1/2}} h_i(y)|\Big)^2 \,\frac{dy\,dt}{t}
\\
&\lesssim
\sum_i \int_{2\,Q_i} \int_{c\,r_i}^{\infty} |t\,\nabla_{y,t} e^{-t\,L^{1/2}} h_i(y)|^2 \,\frac{dy\,dt}{t}
\\
&\lesssim
\sum_i \Big(\int_{c\,r_i}^{\infty}  \int_{\RR^n} |t\,\nabla_y e^{-t\,L^{1/2}} h_i(y)|^2 \,\frac{dy\,dt}{t}
\\
&\qquad\qquad
+
\int_{c\,r_i}^{\infty}  \int_{\RR^n} |t\,L^{1/2} e^{-t\,L^{1/2}} h_i(y)|^2 \,\frac{dy\,dt}{t}\Big)
\\
&\lesssim
\sum_i \int_{c\,r_i}^{\infty}  \int_{\RR^n} |t\,L^{1/2} e^{-t\,L^{1/2}} h_i(y)|^2 \,\frac{dy\,dt}{t},
\end{align*}
where we have used the solution of the Kato conjecture \cite{AHLMcT}
to replace $\nabla_{y}$ by $L^{1/2}$. Next we use the subordination formula \eqref{subordination},
Minkowski's inequality and the change of variable $t\mapsto t':=t^2/4\,s$,
\begin{align*}
&\Big(\int_{c\,r_i}^{\infty}  \int_{\RR^n} |t\,\,L^{1/2} e^{-t\,L^{1/2}} h_i(y)|^2 \,\frac{dy\,dt}{t}\Big)^{1/2}
\\&\qquad
\lesssim
\int_0^\infty e^{-s} \Big(\int_{c\,r_i}^{\infty}  \int_{\RR^n} \frac{t^2}{4\,s}\,|L^{1/2} e^{-\frac{t^2\,L}{4\,s}}h_i(y)|^2 \,\frac{dy\,dt}{t}\Big)^{1/2}\,ds
\\&\qquad
\lesssim
\int_0^\infty e^{-s} \Big(\int_{c\,r_i^2/s}^{\infty}  \int_{\RR^n} |(t\,L)^{1/2} e^{-t\,L}h_i(y)|^2 \,\frac{dy\,dt}{t}\Big)^{1/2}\,ds.
\end{align*}
Next we take $a=\frac{n}{p}-\frac{n}{2}$ and use the square function estimate of McIntosh-Yagi based on $(t\,L)^{(a+1)/2}\,e^{-t\,L}$:
\begin{align*}
&\Big(\int_{c\,r_i}^{\infty}  \int_{\RR^n} |t\,\,L^{1/2} e^{-t\,L^{1/2}} h_i(y)|^2 \,\frac{dy\,dt}{t}\Big)^{1/2}
\\
&\qquad
\lesssim
\int_0^\infty e^{-s} \Big(\int_{c\,r_i^2/s}^{\infty}  \int_{\RR^n} |(t\,L)^{(a+1)/2}\, e^{-t\,L} L^{-a} h_i(y)|^2 \,t^{-a}\,\frac{dy\,dt}{t}\Big)^{1/2}\,ds
\\
&\qquad
\lesssim
\int_0^\infty e^{-s} \Big(\frac{s}{r_i^2}\Big)^{a/2} \,\Big(\int_{\RR^n} \int_{0}^{\infty}  |(t\,L)^{(a+1)/2}\, e^{-t\,L} (L^{-{a/2}} h_i)(y)|^2 \,\frac{dt}{t}\,dy\Big)^{1/2}\,ds
\\
&\qquad
\lesssim
r_i^{-a}\,\|L^{-a/2} h_i\|_2
\lesssim
r_i^{-a}\,\|h_i\|_p
=
r_i^{-a}\,\|(I-A_{r_i})b_i\|_p,
\end{align*}
where we have used \cite[Proposition 5.3]{A} in the last inequality. To conclude we use that $I-A_{r_i}=(I-e^{-r_i^2\,L})^N$ is uniformly bounded on $L^p$ and the CZ lemma
\begin{align*}
I_{\rm loc}
&\lesssim
\sum_i r_i^{-2\,a}\|b_i\|^2_p
\lesssim
\alpha^2\,\sum_i r_i^{-2\,a} \,|Q_i|^{2/p}
\lesssim
\alpha^2\,\sum_i |Q_i|
\\
&\lesssim
\alpha^{2-p}\,\int_{\RR^n}|f(x)|^p\,dx.
\end{align*}

Next we estimate $I_{\rm glob}$. We write $C_j(Q_i)=2^{j+1}\,Q_i\setminus 2^j\,Q_i$, $j\ge 1$.  By duality we can take a function $0\le H\in L^2(\RR^{n+1}_+, \frac{dy\,dt}{t})$ with norm $1$ such that
\begin{align*}
I_{\rm glob}^{1/2}
&
\lesssim
\Big(
\iint_{\RR^{n+1}_+} \Big(\sum_i \chi_{\RR^n\setminus 2\,Q_i}(y)\,|t\,\nabla_{y,t} e^{-t\,L^{1/2}} h_i(y)|\Big)^2\frac{dy\,dt}{t}\Big)^{1/2}
\\
&=
\sum_i\int_0^\infty\int_{\RR^n\setminus 2\,Q_i} |t\,\nabla_{y,t} e^{-t\,L^{1/2}} h_i(y)|\,H(y,t)\frac{dy\,dt}{t}
\\
&\lesssim
\sum_i\sum_{j=1}^\infty 2^{j\,n}|Q_i|\int_0^\infty\aver{C_j(Q_i)} |t\,\nabla_{y,t} e^{-t\,L^{1/2}} h_i(y)|\,H(y,t)\frac{dy\,dt}{t}
\\
&\le
\sum_i\sum_{j=1}^\infty 2^{j\,n}|Q_i|\,
\Big(\int_0^\infty\aver{C_j(Q_i)} |t\,\nabla_{y,t} e^{-t\,L^{1/2}} h_i(y)|^2\,\frac{dy\,dt}{t}\Big)^{1/2}\,
\\
&\hskip3cm\times
\Big(\int_0^\infty\aver{2^{j+1\,Q_i}} H(y,t)^2\frac{dy\,dt}{t}\Big)^{1/2}
\\
&\le
\sum_i\sum_{j=1}^\infty 2^{j\,n}|Q_i|\,I_{ij}\,
\essinf_{y\in Q_i} M\tilde{H}(y)^{1/2}
\end{align*}
where $\tilde{H}(y)=\int_0^\infty H(y,t)^2\,dt/t$. We estimate $I_{ij}$ by the subordination formula, Minkowski's inequality and the change of variable $t\mapsto t':=t^2/4\,s$,
\begin{align*}
I_{ij}
&\lesssim
\int_0^\infty e^{-s}  \Big(\int_0^\infty\aver{C_j(Q_i)} \Big|\frac{t}{\sqrt{4\,s}}\,\nabla_{y} e^{-\frac{t^2\,L}{4\,s}} h_i(y)\Big|^2\,\frac{dy\,dt}{t}\Big)^{1/2}\,ds
\\
&\qquad +
\int_0^\infty e^{-s}  \Big(\int_0^\infty\aver{C_j(Q_i)} \Big|\frac{t}{\sqrt{4\,s}}\,L^{1/2} e^{-\frac{t^2\,L}{4\,s}} h_i(y)\Big|^2\,\frac{dy\,dt}{t}\Big)^{1/2}\,ds
\\
&\lesssim
\int_0^\infty e^{-s}  \Big(\int_0^\infty\aver{C_j(Q_i)} |\sqrt{t} \,\nabla_{y} e^{-t\,L} h_i(y)\Big|^2\,\frac{dy\,dt}{t}\Big)^{1/2}\,ds
\\
&\qquad +
\int_0^\infty e^{-s}  \Big(\int_0^\infty\aver{C_j(Q_i)} \Big|(t\,L)^{1/2} e^{-t\,L} h_i(y)\Big|^2\,\frac{dy\,dt}{t}\Big)^{1/2}\,ds
\\
&\lesssim
\aver{C_j(Q_i)} G_L((I-e^{-r_i^2\,L})^N b_i)(y)^2\,dy
+
\aver{C_j(Q_i)} g_L((I-e^{-r_i^2\,L})^N b_i)(y)^2\,dy
\\
&\lesssim
2^{-j\,n/2}\,4^{-N\,j}\,2^{-j\,(\frac{n}{p}-\frac{n}{2})}\,\Big(\aver{Q_i} |b_i(y)|^p\,dy\Big)^{\frac1p}
\\
&\lesssim
2^{-j\,(2\,N +\frac{n}{p})}\,\alpha,
\end{align*}
where in the next-to-last estimate we have used \cite[pp. 55, 56]{A} and the notation there for $G_{L}, g_{L}$ (the first one is here the same as $G_{h}$) and in the last one the Calder\'on-Zygmund lemma. Choosing $N$ such that $2\,N +\frac{n}{p}-n>0$ we obtain by Kolomogorv's lemma
\begin{align*}
I_{\rm glob}^{1/2}
&
\lesssim
\alpha
\sum_{j=1}^\infty 2^{-j\,(2\,m +\frac{n}{p}-n)}\,\sum_i |Q_i|\,\essinf_{y\in Q_i} M\tilde{H}(y)^{1/2}
\\
&\lesssim
\alpha\,\int_{\cup_i Q_i}M\tilde{H}(y)^{1/2}\,dy
\\
&\lesssim
\alpha\,|\cup_i Q_i|^{1/2}\,\Big(\int_{\RR^n} \tilde{H}(y)\,dy\Big)^{1/2}
\\
&\lesssim
\alpha\,|\cup_i Q_i|^{1/2}\,\Big(\int_{\RR^n}\int_0^\infty H(y,t)^2\,\frac{dt\,dy}{t}\Big)^{1/2}
\\
&\lesssim
\Big(\alpha^{2-p}\,\int_{\RR^n} |f(x)|^p\,dx\Big)^{1/2}.
\end{align*}

Gathering the estimates we have obtained for $I_{\rm loc}$ and $I_{\rm glob}$ we conclude as desired
$$
I\lesssim \alpha^{-2}(I_{\rm loc}+I_{\rm glob})
\lesssim
\alpha^{-2}\,
\alpha^{2-p}\,\int_{\RR^n} |f(x)|^p\,dx
=
\frac{C}{\alpha^p}\,\int_{\RR^n} |f(x)|^p\,dx.
$$
\end{proof}

\subsubsection{$p>2$}  We shall prove a more general result, namely \eqref{eq:m}. Let $m$ be a non-negative integer and set
$$
{\calG}_{m,P}(f)(x)=\left(\iint_{|x-y| <  t}  |t\nabla_{y,t} \big((t^2L)^m e^{-tL^{1/2}}\big)f(y)|^2 \frac
{dydt}{{t}^{n+1}}\right)^{1/2},
$$

We begin with a series of results that are concerned with functions of $L$ in tent spaces. Then, we shall deal with $\calG_{m,P}$.

Consider the notation of \cite[p. 10]{A}. Let  $\phi$ be holomorphic in $\Sigma_{\mu}$, $\mu\in (\omega,\pi/2)$, with $|\phi(\zeta)| \le C(1+|\zeta|)^{-s}$ for some $s>0$, $C<\infty$ and  all $\zeta \in \Sigma_{\mu}$. Consider for $\alpha\in \CC$, with $\Re \alpha>0$, $\varphi_{\alpha}(\zeta)= \frac{\zeta^\alpha}{(1+\zeta)^\alpha}\phi(\zeta)$. Remark that
$$
\frac{\zeta^\alpha}{(1+\zeta)^\alpha}= (1+\zeta^{-1})^{-\alpha}
$$
and since $\zeta\in \Sigma_{\mu}$ implies $\zeta^{-1} \in \Sigma_{\mu}$ and $\arg(1+\zeta^{-1} )\in (-\mu,\mu)$, we have that
$$
\sup_{\zeta\in \Sigma_{\mu}} \left|\frac{\zeta^\alpha}{(1+\zeta)^\alpha}\right| \le e^{\mu |\Im \alpha|}.
$$

Consider the linear operator, a priori defined for $L^2$ functions and valued in $T^2_{2}$,
$$
T_{\alpha}f = (\varphi_{\alpha}(t^2L)f)_{t>0}.
$$
 In the statements below, constants $C$ are allowed to depend on the real part of $\alpha$ but not on its imaginary part.

\begin{lemma} For $\Re \alpha>0$, $T_{\alpha}$ maps $L^p\cap L^2$ to $T^p_{2}$ when $2\le p<p_{+}(L)$ with norm controlled by $Ce^{\mu |\Im \alpha|}$ for any $\mu \in (\omega,\pi/2)$.
\end{lemma}

\begin{proof} It is enough to consider the boundedness of $T_{\alpha}$ for the vertical norm which dominates the conical one, see Proposition \ref{prop:comp}. In this case, this follows from the bounded holomorphic functional calculus on $L^p$ for $2\le p<p_{+}(L)$ combined with Le Merdy's theorem \cite[Theorem 3]{LeM}.
\end{proof}

\begin{lemma} For $\Re \alpha> \frac{n}{2p_{+}(L)}$, $T_{\alpha}$ maps $L^p$ to $T^p_{2}$ when $2\le p\le \infty$ with norm controlled by $Ce^{\mu |\Im \alpha|}$ for any $\mu \in (\omega,\pi/2)$.

\end{lemma}

\begin{proof} For fixed $\alpha$ it is enough to consider the case $p=\infty$ as one can then complex interpolate from \cite{CMS} between $T^2_{2}$ and $T^\infty_{2}$.
  We claim that for any $2< q<p_{+}(L), $ and any ball $B$,
\begin{multline}\label{varphi}
\left(
\frac1{|B|}\,\iint_{B\times (0,r_{B})} |\varphi_{\alpha}(t^2L) f(x)|^2\,\frac{dx\,dt}{t}
\right)^{1/2}
\\
\le
Ce^{\mu|\Im \alpha|}\,\sum_{j=1}^\infty 2^{-j\,(2\Re\alpha-n/q)}\,\left(\aver{2^j\,B}|f(x)|^2\,dx\right)^{1/2}.
\end{multline}
  We postpone the proof of the claim until the end of this subsection. Now the right hand side is dominated by the $L^\infty$ norm of $f$ by using $\Re \alpha> \frac{n}{2p_{+}(L)}$ and choosing $q<p_{+}(L)$ appropriately. Then the supremum over all $B$ of the left hand side is precisely the $T^\infty_{2}$ norm of $T_{\alpha}f$.
\end{proof}

\begin{lemma} For $0<\Re \alpha \le  \frac{n}{2p_{+}(L)}$, $T_{\alpha}$ maps $L^p$ to $T^p_{2}$ when $2\le p< \frac{np_{+}}{n-2p_{+}\Re \alpha}.$
\end{lemma}

\begin{proof}  By a result of Harboure, Torrea, Viviani \cite{HTV},  there is a  linear map $\iota$ which for all $1<p<\infty$ is an isometry from $T^p_{2}$ to a closed subspace of $L^p_{H}$ where $H=L^2(\RR^{n+1}_{+},\frac{dydt}{t^{n+1}})$.  Thus, the maps $\iota \circ T_{\alpha}$ form an analytic family of  linear operators  and they are bounded from $L^p$ to $L^p_{H}$ for $(1/p,\alpha)$ given by the two above lemmas. Stein's complex interpolation theorem (see \cite[Theorem 1.3.7]{Gra}), extended to $H$-valued functions (use the linear $\CC$-valued maps
$f\mapsto \langle \iota \circ T_{\alpha}(f), h \rangle$ for any fixed $h\in H$),   applies since the growth is controlled in  $\Im \alpha$ and gives the desired range of $p$ in terms of $\Re \alpha$.
\end{proof}

We can use the above combined with the following lemma whose proof is postponed to Section \ref{sec:Caccio}

\begin{lemma}\label{lem:Caccio} Let $m$ be a non-negative number. For $C$ depending only on ellipticity and dimension, for any  function $f\in L^2$ 	and any $x\in \RR^n$,
\begin{align*}
{\calG_{m,P}}(f)(x)
&\le
mC \left( \iint_{|x-y| <  2t}  \big|\big((t^2L)^m e^{-t^2 L}f\big)(y)\big|^2 \frac{dydt}{{t\, }^{n+1}}\right)^{1/2}
\\
&\qquad\qquad
+ C\left( \iint_{|x-y| <  2t}  \big|\nabla_{y,t}\big((t^2L)^m e^{-t^2L}f\big)(y)\big|^2 \frac
{dydt}{{t\, }^{n-1}}\right)^{1/2}
\\
&\qquad + C\left( \iint_{|x-y| < 2t}  \big|\big((t^2L)^m(e^{-tL^{1/2}}f(y)-e^{-t^2 L}f\big)(y)\big|^2 \frac
{dydt}{{t\, }^{n+1}}\right)^{1/2}.
\end{align*}
\end{lemma}

We  now conclude for ${\calG}_{P}={\calG}_{0,P}$. Start from the decomposition in the previous lemma and notice that the first term vanishes since $m=0$. The second term is bounded on $L^p$ for $2<p<\infty$ using ${\calG}_{h}$ and rescaling $t\mapsto t^{1/2}$  for the part with $\nabla_{y}$ and the same argument applies for the $\partial_{t}$ part because it picks up one more power of $L$ and one still has good decay in the $L^2-L^2$ off-diagonal estimates.

For the term with $e^{-tL^{1/2}}- e^{-t^2L}$, we apply the third lemma concerned with $T_{\alpha}$ with $\phi(\zeta)= (1+\zeta)^{1/2}\zeta^{-1/2}(e^{-\zeta^{1/2}} - e^{-\zeta})$ and $\alpha=1/2$, which gives $2\le p < \frac{np_{+}(L)}{n-p_{+}(L)}= p_{+}(L)^*$ if $p_{+ }(L)<n$ or  $2<p<\infty$ if $p_{+}(L)\ge n$.

For $m$ positive integer then the third term of the decomposition is estimated as above with $\phi(\zeta)= (1+\zeta)^{m+1/2}\zeta^{-1/2}(e^{-\zeta^{1/2}} - e^{-\zeta})$  and $\alpha=m+1/2$, which gives $2\le p < \frac{np_{+}(L)}{n-(2m+1)p_{+}(L)}$ if $(2m+1)p_{+ }(L)<n$ or  $2<p<\infty$ if $(2m+1)p_{+}(L)\ge n$.  {The first term is as good as the second one, \textit{i.e.,} bounded on $L^p$ for $2<p<\infty$.}
\qed

\begin{remark}
It seems that the order $\zeta^\alpha$ for $\varphi_{\alpha}$ at 0 governs the  $p$ range for boundedness of the conical square function. But if the decay of the off-diagonal estimate is fast enough, then this information is not necessary. For example, consider the conical square function made after $(t^2L)^me^{-t^2L}$ for $m$ a positive real number. When $m$ is an integer, they  are bounded on $L^p$  for all $2<p< \infty$ because the decay is gaussian (polynomial of some high enough degree would suffice). But when $m$ is a  non-integer,  then the decay is polynomial and our method  gives a limited range of $p$ for small $m$ unless $p_{+}(L)=\infty$. In other words, when $p_{+}(L)<\infty$, we obtain a range of $p$ that is discontinuous a function of $m$. We do not know whether this discontinuity is a reality or an artifact of our method. We ask therefore whether  ${\calG}_{m,P}$ is bounded on $L^p$ for $2<p<\infty$ and all  real $m>0$.
\end{remark}

\begin{proof}[Proof of \eqref{varphi}] 
We write $f=f_{\rm loc}+f_{\rm glob}$ where $f_{\rm loc}=f\,\chi_{4\,B}$. Then, using the $L^2$ boundedness of square functions associated with $\varphi_{\alpha}(t^2L)$,
\begin{multline*}
\frac1{|B|}\,\iint_{B\times(0,r_{B})} |\varphi_{\alpha}(t^2L) f_{\rm loc}(x)|^2\,\frac{dx\,dt}{t}
\le
\frac1{|B|}\int_{\RR^n} \Big(\int_0^\infty |\varphi_{\alpha}(t^2L) f_{\rm loc}(x)|^2\,\frac{dt}{t}\Big)\,dx
\\
\le
C\,\frac1{|B|}\int_{\RR^n} |f_{\rm loc}(x)|^2\,dx
=
C\,\aver{4\,B} |f(x)|^2\,dx.
\end{multline*}

It is then enough to show
$$
\bigg(\aver{B} |\varphi_{\alpha}(t^2L) f_{\rm glob}(x)|^2\,dx\bigg)^{1/2}
\le
 Ce^{\mu|\Im \alpha|}\,\frac{t^{2\Re\alpha}}{r_B^{2\Re \alpha}}\,
\sum_{j=2}^\infty 2^{-j\,(2\Re \alpha-n/q)}\,\left(\aver{2^{j+1}\,B}|f(x)|^2\,dx\right)^{1/2}.
$$
Indeed, plugging this estimate in the integral on the Carleson region, we obtain the claim. 

To this end,  we set $f_j=f\,\chi_{C_j(B)}$ with $C_{j}(B)=2^{j+1}B\setminus 2^j B$ so that $f_{\rm glob} = \sum_{j\ge 2} f_{j}$ and by Minkowski's and H\"older's inequalities
$$
\bigg(\aver{B} |\varphi_{\alpha}(t^2L) f_{\rm glob}(x)|^2\,dx\bigg)^{1/2} \le \sum_{j\ge 2} \bigg(\aver{B} |\varphi_{\alpha}(t^2L) f_{j}(x)|^q\,dx\bigg)^{1/q}$$ for any $q\ge 2$.
Fix $j\ge 2$ and use the representations \cite[(2.6)-(2.7)]{A} to estimate $\varphi_{\alpha}(t^2L)f_{j}$. For the
$\eta_{\pm,t}(z)$ given by \cite[(2.7)]{A} we find with $\nu \in (\omega, \mu)$,
$$|\eta_{\pm,t}(z)| \le \frac{Ct^{2\Re\alpha}}{|z|^{\Re \alpha +1} }e^{\nu|\Im \alpha|}.$$
Next, using \eqref{off-heat}  in \cite[(2.6)]{A}  for $e^{-zL}$
 with $p=2$ and $2<q<p_{+}(L)$, $E=C_{j}(B)$ and $F=B$, we easily obtain
$$
\bigg(\aver{B} |\varphi_{\alpha}(t^2L) f_{j}(x)|^q\,dx\bigg)^{1/q} \le
Ce^{\mu|\Im \alpha|}\,\frac{t^{2\Re\alpha}}{r_B^{2\Re \alpha}}\,
2^{-j\,(2\Re \alpha-n/q)}\,\left(\aver{2^{j+1}\,B}|f(x)|^2\,dx\right)^{1/2}.
$$
We see in this last estimate the combined roles of $\Re \alpha$ and $p_{+}(L)$: $\Re \alpha>0$ yields integrability in $t$ while  $2\Re \alpha-n/p_{+}(L)>0$ yields the summability  in space.
\end{proof}

\subsubsection{Converse inequalities}

We basically follow \cite[Theorem 6.1, Step 8]{A}.
What we have proved so far applies to any operator $L$ in our class, and in particular,
to $L=-\Delta$. The explicit formula for the heat kernel implies that $p_-(-\Delta)=1$ and
$p_{+}(-\Delta)=q_+(-\Delta)=\infty$.
Hence, we obtain  the well-known estimates
$$
\|G_{P, -\Delta}f\|_p+ \|G_{h, -\Delta}f\|_p+ \|\calG_{P,-\Delta}f\|_p +\|\calG_{h,-\Delta}f\|_p \lesssim \|f\|_p
$$
for all $1<p<\infty$ and $f\in L^p$, where we have adapted the notation to indicate the operator.

The converse $\|f\|_{p}\lesssim \|G_{h, L}f\|_{p}$ is based on the following  formula for $f,g\in L^2$:
\begin{align*}
\int_{\RR^n} f(x) \, \ol g(x)\, dx &= \lim_{\ep\downarrow 0} \int_{\RR^n} e^{-\ep L}f(x)\, \ol{e^{\ep\Delta}g(x)} \, dx-  \lim_{R\uparrow \infty} \int_{\RR^n} e^{-R L}f(x)\,
\ol{e^{R\Delta}g(x)}\, dx
\\
&
= - \int_0^\infty \frac d {dt} \int_{\RR^n} (e^{-t L}f)(x)\, \ol{(e^{t\Delta}g)(x)}\, dx\, dt
\\
&
= \iint_{\RR^n \times (0,\infty)} (A(x)+I) (\nabla e^{-tL}f)(x) \cdot
 \ol{(\nabla e^{t\Delta}g)(x)}\, dxdt.
\end{align*}
The last equality is
 obtained by integration by parts in the $x$ variable after computing the time
derivative. Hence, we obtain with obvious notation
$$
\left| \int_{\RR^n} f(x) \, \ol g(x)\, dx \right| \le (\|A\|_\infty +1) \int_{\RR^n} G_{h, L}(f) G_{h,-\Delta}(g),
$$
so that
$$
\left| \int_{\RR^n} f(x) \, \ol g(x)\, dx \right| \lesssim   \|G_{h,L}(f)\|_p \|g\|_{p'}
$$
and it follows
$$
\|f\|_p \lesssim \|G_{h,L}(f)\|_p.$$

For $\calG_{h,L}$ the proof is similar. Starting from the equality above, we use the averaging trick of the Introduction and then H\"older's inequality. Details are left to the reader.

For square functions based on the Poisson semigroup, the idea is the same but one needs to integrate by parts in $t$ twice:
\begin{align*}
\int_{\RR^n} f (x)\, \ol g(x)\, dx &= -\int_0^\infty \frac {d} {dt} \int_{\RR^n} (e^{-t L^{1/2}}f)(x)\, \ol{(e^{-t(-\Delta)^{1/2}}g)(x)}\, dx\, dt\\
&
=  \int_0^\infty t \frac {d^2} {dt^2} \int_{\RR^n} (e^{-t L^{1/2}}f)(x)\, \ol{(e^{-t(-\Delta)^{1/2}}g)(x)}\, dx\, dt
\\
&
= \iint_{\RR^n \times (0,\infty)} (A(x)+I) (t\nabla_{x} e^{-tL^{1/2}}f)(x) \cdot
 \ol{(t\nabla_{x} e^{-t(-\Delta)^{1/2}}g)(x)}\, \frac{dxdt}t
 \\
 & \qquad + 2 \iint_{\RR^n \times (0,\infty)}  (t\nabla_{t} e^{-tL^{1/2}}f)(x) \cdot
 \ol{(t\nabla_{t} e^{-t(-\Delta)^{1/2}}g)(x)}\,\frac{dxdt}t.
\end{align*}
The last line is obtained by distributing the second derivatives in $t$ and integrating by parts in $x$ using $\frac {d^2} {dt^2}  (e^{-t L^{1/2}}f)(x)= L(e^{-t L^{1/2}}f)(x)$ and similarly with $-\Delta$.  The two right hand terms are controlled  by both $\|G_{P, L}f\|_{p}\|G_{P,-\Delta}g\|_{p'}$ and  $\|\calG_{P, L}f\|_{p}\|\calG_{P,-\Delta}g\|_{p'}$ so that the conclusion follows as above.

\section{Proof of Lemma \ref{lem:Caccio}}\label{sec:Caccio}
If $m=0$ we take $f\in L^2$ and set $f_0=f$. If $m\ge 1$, as the domain of $L^m$ is dense in $L^2(\RR^n)$, it suffices to assume $f$ in that space and we set $f_{m}=L^mf$. Define $u_{m}=L^me^{-tL^{1/2}}f=  e^{-tL^{1/2}}f_{m}$, $v_{m}=L^me^{-t^2L}f=  e^{-t^2 L}f_{m}$. Notice that
$$
t\nabla_{y,t}(t^{2m}u_{m})
=
2m t^{2m}v_{m}  \vec{e} +2m t^{2m}(u_{m}-v_{m}) \vec{e}
+
t^{2m}(t\nabla_{y,t}u_{m})
$$
with $\vec{e}= (0, \ldots, 0, 1)$.
The first and second terms give rise respectively to the first and third terms on the right hand side of the desired inequality. Therefore it suffices to control the third term which gives a square function that
is pointwise smaller than the integral
$$
I(x)= \iint  |\nabla_{y,t}u_m (y,t)|^2 \varphi^2\bigg( \frac{x-y}t\bigg) \frac{t^{4\,m}dydt}{{t}^{n-1}},
$$
where $\varphi$ is a smooth positive function with $\varphi=1$ on the unit ball $B(0,1)$, supported in the ball $B(0,2)$. To justify the calculations, for $0<r<R/10<\infty$, let $\psi_{r,R}(t)=\zeta(t/r)(1-\zeta(t/R))$ where $\zeta$ is a smooth function that satisfies $0 \le \zeta \le 1$,  $\zeta(t)=0$ if $t\le 1/2$ and $\zeta(t)=1$ if $t\ge 2$ and set
 $$
I_{r,R}(x)= \iint  |\nabla_{y,t}u_m (y,t)|^2 \varphi^2\bigg( \frac{x-y}t\bigg) \psi_{r,R}^2(t) \frac{t^{4\,m}dydt}{{t}^{n-1}}.
$$

Let $B$ be the $(n+1)\times (n+1)$ block matrix with $A$ being one block and 1 the other one. By ellipticity $I_{r,R}(x) \le C(\lambda) \Re \calI_{r,R}(x)$ with
$$
\calI_{r,R}(x)= \iint  B(y)\nabla_{y,t}u_m \cdot \overline{ \nabla_{y,t}u_m}\,  \varphi^2\bigg( \frac{x-y}t\bigg)  \psi_{r,R}^2(t) \frac{t^{4\,m}dydt}{{t}^{n-1}}.
$$
 Next, we write
\begin{multline*}
\calI_{r,R}(x)
=
\iint  B(y)\nabla_{y,t}u_m\cdot  \overline{ \nabla_{y,t}(u_m-v_m)} \, \varphi^2\bigg( \frac{x-y}t\bigg) \psi_{r,R}^2(t)  \frac{t^{4\,m}dydt}{{t}^{n-1}}
 \\
 + \iint  B(y)\nabla_{y,t}u_m \cdot  \overline{ \nabla_{y,t}v_m}\,  \varphi^2\bigg( \frac{x-y}t\bigg) \psi_{r,R}^2(t) \frac{t^{4\,m}dydt}{{t}^{n-1}}
= \calI_{r,R}^1(x) + \calI_{r,R}^2(x).
\end{multline*}
In the last integral, distribute  the product $\varphi\psi$ on each gradient term  and  use Young's inequality with  $\ep$ to obtain a bound
$$
 \|B\|_\infty \ep I_{r,R}(x)
 +
 C\ep^{-1} \iint_{|x-y| <  2t}  |\nabla_{y,t}v_m|^2  \frac{t^{4\,m}dydt}{{t}^{n-1}}.
$$
Using that
$$
t^{2m}(t\nabla_{y,t}v_{m})
=
t\nabla_{y,t}(t^{2m}v_{m})-2m t^{2m}v_{m}  \vec{e}
$$
we can  obtain
\begin{multline*}
\calI_{r,R}^2(x)
\le
 \|B\|_\infty \ep I_{r,R}(x)
 \\
 +
 C\ep^{-1} m \iint_{|x-y| <  2t}  |t^{2m} v_m|^2  \frac{dydt}{{t}^{n+1}}
  +
 C\ep^{-1} \iint_{|x-y| <  2t}  |\nabla_{y,t}(t^{2m }v_m)|^2  \frac{dydt}{{t}^{n-1}}.
\end{multline*}
Note that the first term can be hidden if $\ep$ is small enough independently of $r,R,x$.

For $\calI_{r,R}^1(x)$ we integrate by parts using the equation satisfied by
$u_m$ to obtain
$$
\calI_{r,R}^1(x)=-\iint  B(y)\nabla_{y,t}u_m \cdot  \nabla_{y,t}\bigg\{\frac{t^{4m}}{t^{n-1}}\varphi^2\bigg( \frac{x-y}t\bigg) \psi_{r,R}^2(t) \bigg\} \overline{ (u_m-v_m)} dydt.
$$
Note that
$$
\nabla_{y,t}\bigg\{\frac{t^{4m}}{t^{n-1}}\varphi^2\bigg( \frac{x-y}t\bigg) \psi_{r,R}^2(t) \bigg\}
=
\frac{t^{2m}}{t^{(n-1)/2}}\varphi\bigg( \frac{x-y}t\bigg) \psi_{r,R}(t)  \frac{\theta(y,t)t^{2m}}{t^{(n+1)/2}}
$$
where $\theta\colon \RR^{n+1}_+ \to \RR^{n+1}$ is a  function with support
in the cone defined by $|x-y|\le 2t$ and is bounded independently of $x, r ,R$. Hence, another application of Young's inequality with
$\ep$ yields a bound
$$
\calI_{r,R}^1(x)\le
\|B\|_\infty \ep I_{r,R}(x) + C\ep^{-1} \iint_{|x-y| <  2t}  |u_m-v_m|^2  \frac
{t^{4m}dydt}{{t\, }^{n+1}}.$$
Again, the first term can be hidden if $\ep$ is small enough independently of $r,R,x$.
Gathering the obtained estimates we conclude that
\begin{multline*}
I_{r,R}(x) \le C(n,\lambda, \Lambda) \bigg(
m \iint_{|x-y| <  2t}  |t^{2m} v_m|^2  \frac{dydt}{{t}^{n+1}}
+
\iint_{|x-y| <  2t}  |\nabla_{y,t}(t^{2m }v_m)|^2  \frac{dydt}{{t}^{n-1}}
\\
+
\iint_{|x-y| <  2t}  |t^{2m}(u_m-v_m)|^2  \frac{dydt}{t^{n+1}}\bigg).
\end{multline*}
Letting $r\downarrow 0$ and $R\uparrow \infty$, one obtains the desired estimate.

\end{document}